\documentclass[12pt,a4paper]{amsart}

\usepackage{fullpage}
\usepackage{amssymb}  
\usepackage{latexsym} 
\usepackage{comment}
\usepackage{color}
\usepackage{enumerate}

\usepackage[pdftex]{epsfig}
\newcommand{\Gr}[2]{\psfig{file=#1.pdf,width=#2}}

\usepackage{charter,euler} 
\DeclareSymbolFont{bchoperators}{T1}{bch}{m}{n}
\SetSymbolFont{bchoperators}{bold}{T1}{bch}{b}{n}
\makeatletter
\renewcommand{\operator@font}{\mathgroup\symbchoperators}
\makeatother

\usepackage{titlesec} 
\titleformat{\section}{\normalfont\bfseries\filcenter}{\thesection}{1em}{}
\titleformat{\subsection}{\normalfont\bfseries}{\thesubsection}{1em}{}
\titleformat{\subsubsection}{\normalfont\bfseries}{\thesubsubsection}{1em}{}

\newcommand{\R}{{\mathbb R}}
\newcommand{\fg}{{\mathfrak g}}

\newtheorem{Theorem}{Theorem}
\newtheorem{Lemma}[Theorem]{Lemma}
\newtheorem{Corollary}[Theorem]{Corollary}

\definecolor{darkgreen}{rgb}{0,0.5,0}
\usepackage[
        colorlinks, citecolor=darkgreen,
        backref,
        pdfauthor={Michael Stoll},
        pdftitle={Coarse length can be unbounded in 3-step nilpotent Lie groups},
]{hyperref}
\usepackage[alphabetic,backrefs,lite]{amsrefs} 

\begin{document}

\title[Coarse length in 3-step nilpotent Lie groups]%
      {Coarse length can be unbounded \\ in 3-step nilpotent Lie groups}

\author{Michael Stoll}
\address{Mathematisches Institut,
         Universität Bayreuth,
         95440 Bayreuth, Germany.}
\email{Michael.Stoll@uni-bayreuth.de}

\date{June 13, 2010/July 14, 2025}


\begin{abstract}
  In~\cite{Stoll1998b}, we remarked that, contrary to 2-step nilpotent
  simply connected Lie groups, in 3-step nilpotent simply connected
  Lie groups it is possible that `$\R$-words' in the given generators
  cannot be replaced by an equally long $\R$-word representing the same
  group element and having a bounded number of direction changes.
  In this note, we present an example for this phenomenon.
\end{abstract}

\maketitle


\section{Introduction}

Let $G$ be a (connected and) simply connected nilpotent Lie group
with Lie algebra~$\fg$. Via the exponential map, we can identify $\fg$
with~$G$. We fix a set $S = \{x_1, \dots, x_n\}$ of generators of~$\fg$.
Let $S_{\R} = \{x^t : x \in S, t \in \R\}$.
We consider words of the form
\[ w = x_{i_1}^{t_1} x_{i_2}^{t_2} \cdots x_{i_k}^{i_k} \in S_{\R}^* \,, \]
i.e., words over the alphabet~$S_{\R}$. We will call such words
\emph{$\R$-words over~$S$}.
We define the \emph{length} of~$w$ to be
\[ \ell(w) = \sum_{j=1}^k |t_j| \]
and the \emph{coarse length} of~$w$ to be
\[ \ell'(w) = k \,. \]
We also let $|w|$ denote the result of evaluating $w$ in~$G$. In this
context, we set
$|x^t| = t x$ for $t \in \R$ and $x \in G = \fg$, where the scalar multiplication
with~$t$ is taken in the Lie algebra.
Then we can define the \emph{length} and \emph{coarse length} of
an element $g$ of~$G$ by
\[ \ell(g) = \inf\,\bigl\{\ell(w) : w \in S_{\R}^*, |w| = g\bigr\} \in \R_{\ge 0} \]
and
\[ \ell'(g) = \inf\,\bigl\{\ell'(w) : w \in S_{\R}^*, |w| = g, \ell(w) = \ell(g)\bigr\}
            \in \R_{\ge 0} \cup \{\infty\} \,.
\]

In our paper~\cite{Stoll1998b}*{Lemma~3.3}, we show that in the
case that $G$ is 2-step nilpotent, there is a number~$N$ such that
for every $\R$-word $w \in S_{\R}^*$, there is another $\R$-word $w' \in S_{\R}^*$
satisfying $|w'| = |w|$, $\ell(w') \le \ell(w)$, and $\ell'(w') \le N$.
This implies that $\ell'(g) \le N$ for all $g \in G$.

We then remark that this assertion is not true in general for 3-step
nilpotent simply connected Lie groups. Since various people have asked
for more precise information, we give an explicit example in this note.


\section{The Example}

Consider the free 3-step nilpotent Lie algebra $\fg$ on two generators
$X$ and~$Y$. Then
\[ X, \; Y, \; \tfrac{1}{2} [X,Y], \; \tfrac{1}{12} [X,[X,Y]], \; \tfrac{1}{12} [Y,[Y,X]] \]
is a basis of~$\fg$, which we use to identify $\fg$ with~$\R^5$.
We can also identify $\fg$ with the corresponding Lie group~$G$
(i.e., we take the exponential and logarithm maps to be the identity on
the underlying set); then the group law in~$G$ is given by
\[ a \cdot b = a + b + \tfrac{1}{2} [a,b] + \tfrac{1}{12} [a,[a,b]] + \tfrac{1}{12} [b,[b,a]] \]
in terms of the Lie bracket on~$\fg$.

Now we consider elements of~$G$ that can be written in the form
\[ X^{s_1} Y^{t_1} X^{s_2} Y^{t_2} \cdots X^{s_N} Y^{t_N} \]
with $s_i$, $t_j$ nonnegative and $\sum_i s_i = \sum_j t_j = 1$. In terms of the
basis above, the result will be of the form $(1, 1, u, v, w)$. Let $M$ be
the subset of $\R^3$ consisting of all triples $(u, v, w)$ that can be
obtained in this way.

\begin{Lemma}
  The point $(0, 0, 0)$ is in the closure of~$M$, but not in~$M$ itself.
\end{Lemma}

\begin{figure}[htb]
  \begin{center}
    \Gr{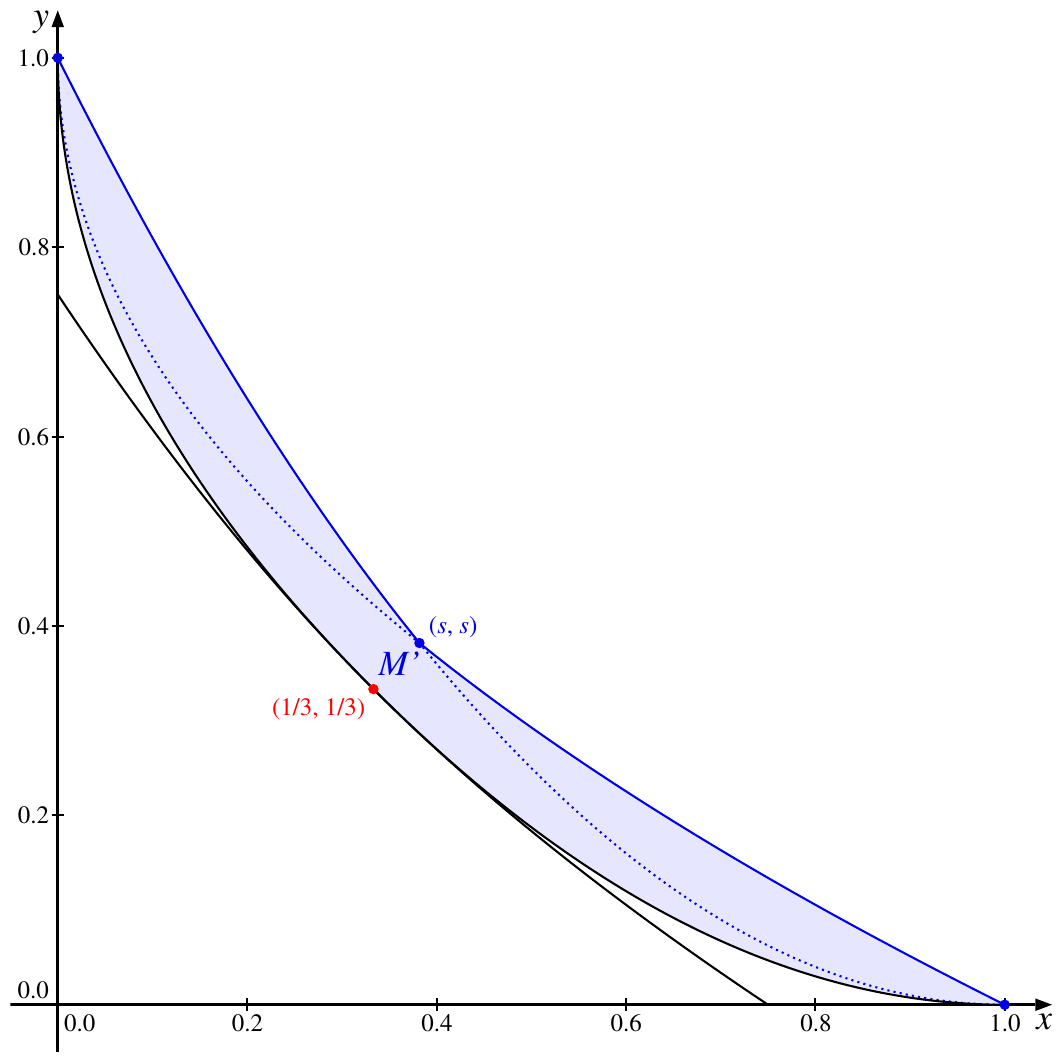}{0.8\textwidth}
  \end{center}
  \caption{The set $M'$. Here $s = (3 - \sqrt{5})/2$.}
  \label{fig1}
\end{figure}

\begin{proof}
  Let $\Sigma$ be the set of $\R$-words
  \[ X^{s_1} Y^{t_1} X^{s_2} Y^{t_2} \cdots X^{s_N} Y^{t_N} \]
  with $s_i$, $t_j$ nonnegative and $\sum_i s_i = \sum_j t_j = 1$ as above,
  and let $\Sigma'$ be its image in the group~$G$ (which corresponds to
  $\{(1,1)\} \times M$).
  For $0 \le t \le 1$, define maps $a_t : \Sigma \to \Sigma$ and
  $b_t : \Sigma \to \Sigma$ by
  \[ a_t(w) = w(X^{1-t}, Y) X^t \qquad\text{and}\qquad b_t(w) = w(X, Y^{1-t}) Y^t \]
  (i.e., for $a_t$, we replace each $X^s$ in~$w$ by~$X^{s(1-t)}$ and multiply
  the result on the right by~$X^t$; similarly for~$b_t$).
  Since each element of~$\Sigma$ can be
  obtained by a succession of applications of these maps to one of the
  words~$X Y$ or $Y X$,
  $\Sigma'$ is the smallest subset of~$G$ containing $X Y$ and~$Y X$ that is
  invariant under these maps. Therefore $M$ is the smallest subset of $\R^3$
  containing $(\pm1, 1, 1)$ that is invariant under the maps induced by
  $a_t$ and~$b_t$ ($0 \le t \le 1$) on $\R^3$, which are
  \begin{align*}
    a_t(u,v,w) &= \bigl((1-t) u - t, \; (1-t)^2 v - 3t(1-t) u + t(2t-1), \; (1-t) w + t\bigr) \\
    b_t(u,v,w) &= \bigl((1-t) u + t, \; (1-t) v + t, \; (1-t)^2 w + 3t(1-t) u + t(2t-1)\bigr) \,.
  \end{align*}

  We now project to the new coordinates
  \[ x = \frac{v + 3u + 2}{6} \qquad\text{and}\qquad y = \frac{w - 3u + 2}{6} \,. \]
  In terms of $x$ and~$y$, we obtain (again abusing notation slightly)
  \begin{align*}
    a_t(x,y) &= \bigl((1-t)^2 x, \; (1-t) y + t\bigr) \\
    b_t(x,y) &= \bigl((1-t) x + t, \; (1-t)^2 y\bigr) \,.
  \end{align*}
  Let $M'$ be the image of~$M$ under this projection. Then $M'$ is the smallest subset 
  of~$\R^2$ containing $(1,0)$ and~$(0,1)$ that is invariant under $a_t$ and~$b_t$
  as above.
  It is then not hard to show that $M' \setminus \{(1,0), (0,1)\}$ is given by the 
  following inequalities:
  \begin{gather*}
     x < 1 \quad\text{and}\quad y < 1 \quad\text{and}\quad
     4x > 3(1-y)^2 \quad\text{and}\quad 4y > 3(1-x)^2  \\
     \text{and}\quad \bigl(x \le (1-y)^2 \quad\text{or}\quad y \le (1-x)^2\bigr) \,.
  \end{gather*}
  (One shows that the inequalities are satisfied for $a_t(1,0)$ when~$t < 1$, 
  similarly for $b_t(0,1)$ when~$t < 1$, and that they are preserved under~$a_t$
  and~$b_t$ when~$t < 1$. This shows the inclusion relevant for us. The other one 
  follows from the fact that all points on the diagonal $x=y$ that satisfy the 
  inequalities can be reached; all other points in the domain can be reached 
  from there.) See Figure~\ref{fig1}, where $s = (3 - \sqrt{5})/2$.

  In particular, $(1/3,1/3)$ is not in~$M'$, so $(0,0,0)$ is not in~$M$.

  On the other hand,
  \[ X + Y = \lim_{n \to \infty} (X^{1/n} Y^{1/n})^n \,, \]
  so $(0,0,0)$ is in the closure of~$M$. In other words, $X + Y$ is in the closure 
  of~$\Sigma'$, but not in~$\Sigma'$.
\end{proof}

This
implies that the coarse length is unbounded on the set of elements of~$G$ in
the form given above. (Since otherwise, $M$ would have to be compact,
compare~\cite{Stoll1998b}.) Note
that the length of a word as above cannot be shortened to be less than~2,
since otherwise we could not reach the projection~$(1, 1)$ in~$G/G'$
(where $G'$ is the derived subgroup of~$G$; the projection to~$G/G'$
is the projection to the first two coordinates).

\begin{Corollary}
  Let $g_0 \in G$ correspond to $(1,1,0,0,0)$ with respect to the basis above,
  i.e., $g_0 = X + Y$ as an element of~$\fg$.
  Then $\ell'(g_0) = \infty$. (I.e., there is no geodesic joining the origin
  to~$g_0$ consisting of finitely many steps in directions given by the
  generators.)
\end{Corollary}

\begin{proof}
  It is clear that $\ell(g_0) \ge 2$, since for the element~$g \in G$ corresponding to
  $(s,t,u,v,w)$, we always have $\ell(g) \ge |s|+|t|$. We claim that
  $\ell(g_0) = 2$. This follows from the facts that $\ell$ is continuous
  on~$G$, that $\ell$ takes the value~2 on~$\Sigma'$ and that
  $g_0$ is in the closure of~$\Sigma'$.
  Now, since $g_0$ is not in~$\Sigma'$, by definition of~$\Sigma'$ there is no
  $\R$-word~$w$
  of length~2 that represents~$g_0$, hence the set in the definition of~$\ell'(g_0)$
  is empty.
\end{proof}

\begin{Corollary}
  For every $N$, there exists an $\R$-word $w \in \Sigma$ such that there is
  no $\R$-word $w'$
  satisfying $\ell(w') \le 2 = \ell(w)$, $|w'| = |w|$ and $\ell'(w') \le N$.
\end{Corollary}

\begin{proof}
  If $g \in \Sigma'$, then any $\R$-word representing~$g$ must have length~$\ge 2$,
  and any $\R$-word of length~2 representing~$g$ must be in~$\Sigma$. The previous
  corollary implies that as we let $g$ approach~$g_0$ within~$\Sigma'$, the minimal
  coarse length of an $\R$-word in~$\Sigma$ representing~$g$ tends to infinity.
\end{proof}


\end{document}